\documentclass[11pt]{amsart}

\usepackage{texmac,semtkzX}

\def\cO{{\mathcal{O}}}

\title[Hyperelliptic curves over $2$-adic fields with ordinary reduction]{A note on hyperelliptic curves with ordinary reduction over $2$-adic fields} 

 \author{Vladimir Dokchitser}
\author{Adam Morgan}

\address{University College London, London WC1H 0AY, UK}
\email{v.dokchitser@ucl.ac.uk}

\address{ University of Glasgow, Glasgow, G12 8QQ.}
\email{adam.morgan@glasgow.ac.uk}
 
\begin{document}

\maketitle  

\begin{abstract}
We study a  class of semistable ordinary hyperelliptic curves over $2$-adic fields and the special fibre of their minimal regular model. We show that these curves can be controlled using `cluster pictures', similarly to the case of odd residue characteristic. 
\end{abstract}


\section{Introduction} 

Let $K$ be a finite extension of $\mathbb{Q}_p$ and $C/K:y^2=f(x)$ be a hyperelliptic curve of genus $g\geq 2$. If $p$ is odd then the reduction of $C$ can often be described explicitly in terms of the $p$-adic distances between the roots of $f(x)$ (see e.g. \cite{DDMM18,MR4201122,MR3576328,MR1671741,PS2019} and the references therein). By contrast, when $p=2$ these results, which all exploit the fact that the `$x$-coordinate' morphism $C\rightarrow \mathbb{P}^1_K$ has degree coprime to $p$, no longer apply. 
More naively, one can already see that extra difficulties must arise upon noting that no short form Weierstrass equation ever defines a smooth curve over a field of characteristic $2$; at the very least one must instead work with   Weierstrass equations of the shape $y^2+Q(x)y=P(x)$ to have a hope of giving explicit equations for components of the reduction.  

In spite of the above, for applications to global problems it is often desirable to represent a hyperelliptic curve over $\mathbb{Q}$ (or some other number field) by a short form Weierstrass equation and say something about its reduction at all primes $p$. This is the situation in \cite{DM2019}, for example, in which the $2$-parity conjecture is proven for a broad class of Jacobians of genus $2$ curves, ultimately via a comparison of their local invariants. The principal aim of the present article is to produce a reasonable stock of examples for which such local comparisons may readily be carried out at $p=2$. As we shall see, this includes in particular all curves having good ordinary reduction.

\subsection{Weierstrass equations of ordinary curves} 
In what follows, $K$ denotes a finite extension of $\mathbb{Q}_2$  with residue field $k$ and ring of integers $\mathcal{O}_K$. We denote by $K^{nr}$ the maximal unramified extension of $K$. Let $C/K$ be a hyperelliptic curve of genus $g\geq 2$ and consider the following property:
\begin{notation}[Property $(\star)$] \label{notat_star_intro}
We say a Weierstrass equation $y^2=cf(x)$ for $C$ satisfies $(\star)$ if:   $c\equiv 1\pmod 4$, $f(x)\in \mathcal{O}_K[x]$ is monic of degree $2g+2$ and squarefree, and the roots of $f(x)$ can be put into pairs $\{\alpha_1,\beta_1 \}$, ..., $\{\alpha_{g+1},\beta_{g+1}\}$ satisfying
\begin{itemize}
\item \quad  $(x-\alpha_i)(x-\beta_i)\in K^{nr}[x]$ for all $i$,
\item   \quad$v(\alpha_i\!-\!\beta_i)= v(4)$ for all $i$,
\item \quad
$
v(\alpha_i\!-\!\alpha_j)=v(\beta_i\!-\!\beta_j)=v(\alpha_i\!-\!\beta_j)=0 
$ for all $i\neq j$.
\end{itemize}
\end{notation}

Our first result is then as follows:

\begin{theorem}    \label{main:ordinary}
If $C$ can be represented by a Weierstrass equation satisfying $(\star)$ then $C$ has good ordinary reduction. Conversely, if $|k|\geq g+1$ and $C$ has good ordinary reduction, then it can be represented by a Weierstrass equation satisfying $(\star)$.
\end{theorem}

 We will prove Theorem \ref{main:ordinary} as part of a more general statement. To describe this we introduce the following additional properties:
 
\begin{notation}[Properties $(\star \star)$ and $(\dagger)$] \label{double_star_notat}
We say that a Weierstrass equation $y^2=cf(x)$ for $C$ satisfies $(\star\star)$ if it satisfies the property $(\star)$ above but with the second bullet point   replaced by the weaker condition  
\begin{itemize}
\item \quad $v(\alpha_i-\beta_i)\geq v(4)$ for all $i$.
\end{itemize}

We say that $C$ has reduction type $(\dagger)$ if $C/K$ has semistable reduction and the geometric special fibre of its stable model   is either:  
\begin{itemize}
\item[$\bullet$]  irreducible   with normalisation an ordinary curve, or 
\item[$\bullet$]   a union of $2$ rational curves intersecting transverally at $g+1$ points.  
\end{itemize}
\end{notation}
%

We then have the following result relating properties $(\star \star)$ and $(\dagger)$:

 \begin{theorem} \label{main:beyond_ordinary}
If $C$ can be represented by a Weierstrass equation satisfying $(\star \star)$ then $C$ has reduction type $(\dagger)$. Conversely, if $|k|\geq g+1$ and $C$ has reduction type $(\dagger)$, then it can be represented by a Weierstrass equation satisfying $(\star \star)$. 
 \end{theorem}

\subsection{Special fibre of the stable and minimal regular models}
When the conditions of Theorem \ref{main:beyond_ordinary} are satisfied, one can moreover read off the precise structure of the stable and minimal regular models of $C$ in a straightforward fashion. 

\begin{proposition} \label{explicit_reduction_prop}
Suppose that $C$ is given by a Weierstrass equation satisfying $(\star \star)$. Let $\mathcal{C}/\mathcal{O}_K$ denote the stable model of $C$, and let $\mathcal{C}'=\mathcal{C}\times_{\mathcal{O}_K}\mathcal{O}_{K^{nr}}$. Then the special fibre of $\mathcal{C}'$ has one ordinary double point $P_i$ for each pair $\{\alpha_i,\beta_i\}$ with $v(\alpha_i-\beta_i)>v(4)$, and no others. Such a point $P_i$ has thickness $2(v(\alpha_i-\beta_i)-v(4))$ in $\mathcal{C}'$. 
\end{proposition}

\begin{remark} \label{thickness_remark_intro}
See e.g. \cite[Definition 10.3.23]{MR1917232} for the definition of the thickness of an ordinary double point. By \cite[Corollary 10.3.25]{MR1917232} the minimal proper regular model of $C$ over $\mathcal{O}_{K^{nr}}$ is obtained from the stable model by replacing each ordinary double point of thickness $n$ by a chain of $n-1$ rational curves intersecting transversally. 

We remark also that in the context of the above proposition, the geometric special fibre of the stable model is a union of $2$ rational curves intersecting transversally if and only if there are precisely $g+1$ pairs or roots $\{\alpha_i,\beta_i\}$ with $v(\alpha_i-\beta_i)>v(4)$.
\end{remark}

Note that Theorem \ref{main:ordinary} follows immediately from Theorem \ref{main:beyond_ordinary} and Proposition \ref{explicit_reduction_prop}.

 
\begin{remark} \label{rem:clusters}
Property $(\star \star)$ can be rephrased in the language of cluster pictures \cite{DDMM18}. Specifically, let $y^2=cf(x)$ be a Weierstrass equation for $C$ with $c\equiv 1\pmod 4$ and $f(x)\in \mathcal{O}_K[x]$ monic. Then this equation satisfies $(\star \star)$  if and only if $f(x)$ has cluster picture 
\begin{center}
		\clusterpicture            
  \Root[A] {} {first} {r1};
  \Root[A] {} {r1} {r2};
  \Root[A] {6} {r2} {r3};
  \Root[A] {} {r3} {r4};
  \Root[A] {6} {r4} {r5};
  \Root[A] {} {r5} {r6};
  \Root[Dot] {5} {r6}{r7};
  \Root[Dot] {} {r7}{r8};
  \Root[Dot] {} {r8}{r9};
   \Root[A] {4} {r9}{r10};
    \Root[A] {} {r10}{r11};
  \ClusterLDName c1[][n_1][ ] = (r1)(r2);
    \ClusterLDName c2[][n_2][ ] = (r3)(r4);
      \ClusterLDName c3[][n_3][ ] = (r5)(r6);
      \ClusterLDName c4[][n_{g+1}][ ] = (r10)(r11);
  \ClusterLDName c4[][0][] = (c1)(c3)(c3)(r7)(r8)(r9)(c4);
\endclusterpicture 
\end{center}
\noindent with each $n_i\geq v(4)$ and with the inertia group of $K$ acting trivially on all proper clusters. Informally, Proposition \ref{explicit_reduction_prop} then says that a qualitative description of the geometric special fibre of the minimal  regular model  can be obtained from such a cluster picture by subtracting $v(4)$ from each $n_i$, and then proceeding as one would in odd residue characteristic (i.e. as described in \cite{hyperuser,DDMM18}, say). 

It seems plausible that a similarly straightforward `cluster picture' description of the reduction of $C$ holds (at least) whenever $C$ is semistable and the special fibre of its stable model has normalisation a disjoint union of ordinary curves. Beyond that the  situation is more delicate; see \cite{MR1962052,MR2272976} for a discussion of the case where $f(x)$ has $2$-adically equidistant roots (i.e. when there is a single proper cluster). 
\end{remark}

\begin{example}
Consider the genus $2$ hyperelliptic curve 
 \begin{equation} \label{global example}
C/\mathbb{Q}:y^2=(x-2)(x+2)(x^2+7x-1)(x^2-9x+7).
\end{equation}
 We will determine its reduction over $\mathbb{Q}_p$ for all primes $p$. Denote by $f(x)$ the right hand side of \eqref{global example}. Computing the discriminant of $f(x)$ we see that $C$ has good reduction away from $\{2,7,11,17,29,53\}$. For odd primes in this set $f(x)$ has the following cluster picture (see \cite[Definition 1.1]{DDMM18}):

\begin{center}
		\clusterpicture            
  \Root[A] {} {first} {r1};
  \Root[A] {} {r1} {r2};
  \Root[A] {} {r2} {r3};
  \Root[A] {} {r3} {r4};
  \Root[A] {2} {r4} {r5};
  \Root[A] {} {r5} {r6};
  \ClusterLDName c1[][1][ ] = (r5)(r6);
  \ClusterLDName c4[][0][] = (r1)(r2)(r3)(r4)(c1);
\endclusterpicture 
\quad \quad\quad
		\clusterpicture            
  \Root[A] {} {first} {r1};
  \Root[A] {} {r1} {r2};
  \Root[A] {2.5} {r2} {r3};
  \Root[A] {} {r3} {r4};
  \Root[A] {5} {r4} {r5};
  \Root[A] {} {r5} {r6};
  \ClusterLDName c1[][1][ ] = (r3)(r4);
  \ClusterLDName c2[][1][ ] = (r5)(r6);
  \ClusterLDName c4[][0][] = (r1)(r2)(c1)(c2);
\endclusterpicture 
\quad \quad\quad
		\clusterpicture            
  \Root[A] {} {first} {r1};
  \Root[A] {} {r1} {r2};
  \Root[A] {2.5} {r2} {r3};
  \Root[A] {} {r3} {r4};
  \Root[A] {5} {r4} {r5};
  \Root[A] {} {r5} {r6};
  \ClusterLDName c1[][\tfrac{1}{2}][ ] = (r3)(r4);
  \ClusterLDName c2[][\tfrac{1}{2}][ ] = (r5)(r6);
  \ClusterLDName c4[][0][] = (r1)(r2)(c1)(c2);
\endclusterpicture 
\end{center}
\vspace{-5pt}
\[p\in \{7,17,29\}\quad\quad  \quad \qquad p=11 \qquad \quad\quad \quad \quad \qquad p=53\qquad\]
 Here each  \smash{\raise4pt\hbox{\clusterpicture\Root[A]{}{first}{r1};\endclusterpicture}} represents a root of $f(x)$. With $p=7$, for example, the picture indicates that there are two roots whose difference has $7$-adic valuation $1$ --- these roots form a `cluster' of size $2$ --- whilst for any other pair of roots $r,r'$ of $f(x)$ we have $\textup{ord}_7(r-r')=0$. To see this is the case note that   $f(x)\pmod 7$ splits completely and has a single double root at $x=2$. This double root accounts for the cluster of size $2$, consisting of $x=2$ and the unique root of $x^2-9x+7$ congruent to $2 \pmod 7$. The valuation of the difference of these  roots is $1$ since $\textup{ord}_7(2^2-9\cdot 2 +7)=1$.

It follows e.g. from \cite[Theorem 1.11]{DDMM18} that the  geometric special fibres of the corresponding minimal regular models take the following form:
\medskip 
\begin{center} \includegraphics[angle=0,scale=0.35]{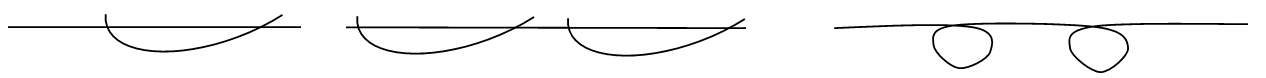} \end{center}
\[p\in \{7,17,29\} \quad\quad \quad \quad\quad   p=11 \quad \quad \quad \quad \quad\quad\quad  \quad p=53\quad \quad\quad \quad \vspace{5pt}\]
That is, $C$ has Namikawa--Ueno type (as set out in \cite{MR0369362}): $I_{2-0-0}$ when $p\in \{7,17,29\}$,  $I_{2-2-0}$ when $p=11$, and $I_{1-1-0}$ when $p=53$.

To determine the reduction at $p=2$ where the results of \cite{DDMM18} no longer apply, note that we can rewrite $f(x)$ as 
\[f(x)=\left(x^2-4\right)\left((x-\varphi_+)^2-16\right)\left((x-\varphi_-)^2-16\right)\]
where $\varphi_{\pm}=\frac{1}{2}(1\pm\sqrt{53})$. With the roots paired as in the quadratic factors above we see  that the given Weierstrass equation for $C$ satisfies $(\star \star)$ over $\mathbb{Q}_2$. From Theorem \ref{main:beyond_ordinary} and Proposition \ref{explicit_reduction_prop} we see that $C$ has reduction type $(\dagger)$ over $\mathbb{Q}_2$, and that the geometric special fibre of the minimal regular model has the same form as depicted for $p=11$ above, thus $C$ has Namikawa--Ueno type $I_{2-2-0}$ at $p=2$. In fact, one sees similarly that $f(x)$ has cluster picture
\begin{center}
		\clusterpicture            
  \Root[A] {} {first} {r1};
  \Root[A] {} {r1} {r2};
  \Root[A] {6} {r2} {r3};
  \Root[A] {} {r3} {r4};
  \Root[A] {6} {r4} {r5};
  \Root[A] {} {r5} {r6};
  \ClusterLDName c1[][2][ ] = (r1)(r2);
    \ClusterLDName c2[][3][ ] = (r3)(r4);
      \ClusterLDName c3[][3][ ] = (r5)(r6);
  \ClusterLDName c4[][0][] = (c1)(c3)(c3);
\endclusterpicture 
\end{center}
so the above is consistent with Remark \ref{rem:clusters}: substracting $\textup{ord}_2(4)=2$ from the depths of the $3$ clusters of size $2$ yields the same cluster picture as for $p=11$.
\end{example}

\subsection{Frobenius action on the special fibre}
Suppose $C$ is given by a Weierstrass equation $y^2=cf(x)$ satisfying $(\star \star)$. One can then complement Proposition \ref{explicit_reduction_prop} with an explicit description of the $\textup{Gal}(\bar{k}/k)$-action on the dual graph of the geometric special fibres of the stable and minimal regular models. This additional information is what is needed to determine local invariants of the Jacobian of $C$ such as its Tamagawa number and root number (see e.g. \cite[Theorem 2.20, Lemma 2.22]{DDMM18} for a recipe for computing these invariants from this data).

Recall from Proposition \ref{explicit_reduction_prop} that to each pair of roots $\{\alpha_i,\beta_i\}$ of $f(x)$ with $v(\alpha_i-\beta_i)>4$, there corresponds an ordinary double point $P_i$ on the geometric special fibre of the stable model of $C$. In Proposition \ref{prop:frob_action_intro} below we describe the $\textup{Gal}(\bar{k}/k)$-action on the points $P_i$, and give an explicit characterisation of when such a point is a \textit{split} ordinary double point in the sense of \cite[Definition 10.3.8]{MR1917232}. This data is sufficient to determine the $\textup{Gal}(\bar{k}/k)$-action on the dual graph of stable model, as explained in e.g.  \cite[Section 2.1]{DDMM18}. To obtain the corresponding action for the minimal regular model, one can then use Remark \ref{thickness_remark_intro}.   

\begin{notation} \label{Weierstrass_notat_intro}
In what follows, for $w\in \mathcal{O}_{K^{nr}}$  we denote by $\overline{w}$ the reduction of $w$ to the residue field $\bar{k}$. With the roots $\{\alpha_1,\beta_1\}$, ..., \{$\alpha_{g+1},\beta_{g+1}\}$ paired as in Notation \ref{double_star_notat},  define 
\[\quad \quad \quad \gamma_i=\frac{\alpha_i+\beta_i}{2}\quad \textup{ and }\quad \eta_i=\big(\frac{\alpha_i-\beta_i}{4}\big)^2~~\quad \quad (1\leq i\leq g+1).\]
By assumption we have $\eta_i\in \mathcal{O}_{K^{nr}}$, and since $\alpha_i+\beta_i\equiv \alpha_i - \beta_i \equiv 0 \bmod 2$ we have $\gamma_i\in \mathcal{O}_{K^{nr}}$ also. Further, define 
\[a=\frac{1}{4}(c-1)\in \mathcal{O}_K\quad \textup{ and }\quad r_i=\overline{\gamma_i}\in \bar{k}.\]
\end{notation}

 We then have the following:

\begin{proposition} \label{prop:frob_action_intro}
In the notation above, the correspondence  $\{\alpha_i,\beta_i\} \mapsto P_i$ is equivariant for the action of $\textup{Gal}(K^{nr}/K)=\textup{Gal}(\bar{k}/k)$.  Further, $P_i$ is a split ordinary double point over $k(P_i)=k(r_i)$ if and only if  
\begin{equation} \label{split_d_p_invariant_intro}
\textup{Trace}_{k(r_i)/\mathbb{F}_2}\Big(\overline{a}+\sum_{j\neq i}\overline{\eta_j}(r_i-r_j)^{-2}\Big)=0.
\end{equation}
\end{proposition}
 
\begin{example} \label{ex:example_genus_2} 
 Consider the genus $2$ hyperelliptic curve 
 \[C/\mathbb{Q}_2:y^2=5(x^2-8)(x^2-7x+13)(x^2+9x+21).\]
 This has roots $\{\pm2\sqrt{2}, \zeta_3\pm 4, \zeta_3^2\pm 4 \}$ where $\zeta_3$ is a primitive $3$rd root of unity, and cluster picture 
 \begin{center}
		\clusterpicture            
  \Root[A] {} {first} {r1};
  \Root[A] {} {r1} {r2};
  \Root[A] {7} {r2} {r3};
  \Root[A] {} {r3} {r4};
  \Root[A] {7} {r4} {r5};
  \Root[A] {} {r5} {r6};
  \ClusterLDName c1[][\frac{5}{2}][ ] = (r1)(r2);
    \ClusterLDName c2[][3][ ] = (r3)(r4);
      \ClusterLDName c3[][3][ ] = (r5)(r6);
        \ClusterLDName c4[][0][] = (c1)(c2)(c3);
\endclusterpicture 
\end{center}
 We see from Proposition \ref{explicit_reduction_prop}  (cf. also Remark \ref{rem:clusters}) that $C$ is semistable and  that the minimal proper regular model of $C$ over $\mathbb{Z}_2^{nr}$ has special fibre
 
\begin{center} \includegraphics[angle=0,scale=0.4]{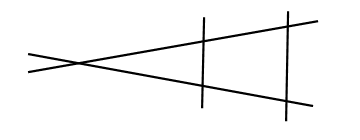} \end{center}

\noindent  where each irreducible component is a rational curve.

We can use Proposition \ref{prop:frob_action_intro} to determine the Frobenius action on this special fibre. With the roots ordered as written, in the terminology of Notation \ref{Weierstrass_notat_intro} we have $a=1$,
\[\gamma_1=0,\quad \gamma_2=\zeta_3, \quad \gamma_3=\zeta_3^2, \quad \eta_1=2,\quad \eta_2=\eta_3=4.\]
Since $\gamma_2$ and $\gamma_3$ are swapped by Frobenius, so are the two components drawn vertically. Further, since $\overline{\eta_i}=0$ for each $i$ whilst $\overline{a}=1$, we see from \eqref{split_d_p_invariant_intro} that the pair of roots $\{\pm 2\sqrt{2}\}$ corresponds to a non-split ordinary double point. Thus Frobenius swaps the two components drawn horizontally also. 

We remark that in fact, as one can show using Remark \ref{explicit_reduction_formulae_rem}, the special fibre of the stable model of $C$ over $\mathbb{Z}_2$ is given explicitly by the equation
\[\left(y+\zeta_3(x^3+x^2+x)\right)\left(y+\zeta_3^2(x^3+x^2+x)\right)=0,\]
whose two (geometric) components are visibly swapped by Frobenius. 
\end{example}

\begin{remark}
In the above example we conclude that  $C(\mathbb{Q}_2)=\emptyset$, since each irreducible component of the special fibre has an even-sized $\textup{Gal}(\bar{\mathbb{F}}_2/\mathbb{F}_2)$-orbit.  In fact, arguing similarly, we have $C(K)=\emptyset$ for every odd degree extension $K/\mathbb{Q}_2$. Thus $C$ is \textit{deficient} in the sense of \cite[Section 8]{MR1740984}.
\end{remark}

\subsection{Reduction map on $2$-torsion} \label{2-tors_reducton_intro} 
Suppose that $C$ is  given by a Weierstrass equation of the form $(\star)$, hence has good ordinary reduction by Theorem \ref{main:ordinary}. Then the  N\'{e}ron model  $\mathcal{J}/\mathcal{O}_K$ of the Jacobian $J/K$ of $C$ is an abelian scheme,  and we have an associated reduction map $J[2]\rightarrow \overline{J}(\bar{k})[2]$ on $2$-torsion points. Here $\overline{J}$ denotes the special fibre of $\mathcal{J}$. Motivated by applications to the parity conjecture (see \cite[Theorem A.1]{DM2019}, for example) we record an explicit description of the kernel of this reduction map in terms of the roots of $f(x)$, the set of which we denote $\mathcal{R}$. 

Recall that  $J[2]$ can be identified with the collection of even-sized subsets $S\subseteq \mathcal{R}$, with addition corresponding to symmetric difference, and along with the relations identifying $S$ with $\mathcal{R}\setminus S$ for each such $S$.  This correspondence is realised explicitly by sending $S\subseteq \mathcal{R}$ to the class of the  divisor 
\begin{equation} \label{eq:2-tors_divisor_intro}
D_S=\sum_{r\in S}P_r-\frac{|S|}{2}(\infty^+ +\infty^- ),
\end{equation}
where $P_r=(r,0)$. For a proof see e.g. \cite[Section 5.2.2]{MR2964027}. The result is now the following:
 
\begin{proposition} \label{prop:reduction_map_2_tors}
With the notation above, and with the roots of $f(x)$ paired as in Notation \ref{notat_star_intro}, the kernel of the reduction map $J[2]\rightarrow \overline{J}(\bar{k})[2]$ is generated by the subsets $\{\alpha_i,\beta_i\}$ for $1\leq i\leq g+1$.
\end{proposition}

\subsection{Layout} 
In Section \ref{sec:Weierstrass_char_2} we review some properties of Weierstrass equations over fields of characteristic $2$ which will be used later.  Section \ref{sec:main_proofs} then proves all results mentioned above, beginning with Theorem \ref{main:beyond_ordinary} and Proposition \ref{explicit_reduction_prop}. For explicit equations defining the special fibre of the stable model when $C$ is given by a Weierstrass equation satisfying $(\star \star)$, see Remark \ref{explicit_reduction_formulae_rem}. 

\subsection{Notation and conventions}
The following notation and conventions will be used throughout the the paper. 

By a \textit{hyperelliptic curve} over a field $F$ we mean a smooth proper geometrically connected curve $C/F$ equipped with a finite separable $k$-morphism $C\rightarrow \mathbb{P}_F^1$ of degree $2$.  We say that $C$ is ordinary if its Jacobian is (however `good' in Theorem \ref{main:ordinary} refers to good reduction of the curve rather than the weaker property of good reduction of its Jacobian).

Let $R$ be a commutative ring. By a \textit{Weierstrass equation} over $R$ we mean an equation  
\begin{equation}  \label{eq:weierstrass}
y^2+Q(x)y=P(x)
\end{equation}  
where $P(x), Q(x)\in R[x]$ are polynomials with $\max\{2\deg Q, \deg P\}\in \{2g+1,2g+2\}$ for some $g\geq 2$. By the scheme $X$ defined by this Weierstrass equation   we mean the $R$-scheme given by gluing the affine charts \eqref{eq:weierstrass} and 
 \begin{equation} \label{eq:chart_at_infty}
 z^2+t^{g+1}Q(1/t)z=t^{2g+2}P(1/t)
 \end{equation}
  via the change of variables $x=1/t$ and $t^{g+1}y=z$.
  
  For a nonarchimedean local field $K$, ring of integers $\mathcal{O}_K$ and residue field $k$, given $w\in \mathcal{O}_K$  we denote by $\overline{w}$ the reduction of $w$ to $k$. For a polynomial $Q(x)\in \mathcal{O}_K[x]$ we denote by $\overline{Q}(x)\in k[x]$ the reduced polynomial.
  

 \section{Weierstrass equations in characteristic $2$} \label{sec:Weierstrass_char_2}
 
 Let $F$ be an algebraically closed field of characteristic $2$. Let $g\geq 2$ and let $C$ be the curve defined by a Weierstrass equation \eqref{eq:weierstrass} over $F$ with $\deg Q =g+1$.
  Denote by $\mathcal{R}$ the set of roots of $Q(x)$ in $F$. The map $(x,y) \mapsto x$ defines  a finite separable morphism $C\rightarrow \mathbb{P}^1_F$ of degree $2$ which ramifies precisely at the points $P_r=(r,\sqrt{P(r)})$ for $r\in \mathcal{R}$. There are  $2$ points lying over the point at infinity on $\mathbb{P}^1_F$; we denote these   $\infty^\pm$.

 \subsection{Smooth Weierstrass equations}
 
 Suppose the curve $C$ is smooth (this can, for example, be detected by the discriminant of the Weierstrass equation; see e.g.  \cite{MR1195511} or \cite[Section 2]{MR1363944}). Then $C$ is a genus $g$ hyperelliptic curve over $F$. Conversely, any hyperelliptic curve $X/F$ of genus $g\geq 2$ can be given by a Weierstrass equation of the form \eqref{eq:weierstrass} with $\deg Q=g+1$ (cf. \cite[Proposition 7.4.24]{MR1917232}). Denote by $J$ the Jacobian of $C$.
 
 \begin{lemma} \label{lem:2-tors_char_2}
The group $J(F)[2]$ can be identified with the collection of even-sized subsets $S\subseteq \mathcal{R}$, with addition corresponding to symmetric difference. Explicitly, $S\subseteq \mathcal{R}$  corresponds to the class of the divisor 
\[D_S=\sum_{r\in S}P_r-\frac{|S|}{2}(\infty^++\infty^-).\]  
\end{lemma}

\begin{proof}
This is  well known; see, for example,  the proof of \cite[Theorem 23]{MR3290950} (cf. also  work of Elkin--Pries \cite{MR3095219}). 
\end{proof}

\begin{remark} \label{rem:size_of_2_tors_char_2}
By Lemma \ref{lem:2-tors_char_2} we have $\dim J(F)[2]= |\mathcal{R}|-1$. In particular, $C$ is ordinary if and only if $Q(x)$ is separable. Whilst we have a running assumption that $g\geq 2$, we note that this equivalence also holds when $g=1$, i.e. for Weierstrass equations $y^2+Q(x)y=P(x)$ where $\deg Q=2$ and $\deg P\leq 4$. 
 \end{remark}
 
 \subsection{Semistable Weierstrass equations} \label{sec:ss_weier_char_2}
 
 We continue to suppose that $C$ is given by a Weierstrass equation \eqref{eq:weierstrass}, but do not assume that $C$ is smooth.
 
 \begin{lemma} \label{sep_Q_char_2}
The curve $C$ is semistable with normalisation a disjoint union of ordinary curves if and only if $Q(x)$ is separable.
 \end{lemma}
 
 \begin{proof}
From the Jacobian criterion one sees that $C$ is smooth away from the points $P_r$ for $r\in \mathcal{R}$, and that such a point  is smooth  if and only if $P'(r)^2+P(r)Q'(r)^2\neq 0$. Moreover, when $P'(r)^2+P(r)Q'(r)^2=0$ the point $P_r$ is an ordinary double point if and only if $r$ is a simple root of $Q(x)$.   For such a point $P_r$,  consider the curve $C_1$ with  Weierstrass equation 
\[w^2+\frac{Q(x)}{x-r}w=\frac{P(x)+P(r)+Q(x)\sqrt{P(r)}}{(x-r)^2} \]
(our assumptions force each side of this equation to lie in $F[x,w]$).  The  morphism $C'\rightarrow C$ defined by 
\begin{equation} \label{eq:partial_normalisation}
(x,w)\longmapsto \left(x,(x-r)w+\sqrt{P(r)}\right)
\end{equation}
realises $C'$ as the partial normalisation of $C$ at the point $P_r$. 

Suppose that $C$ has $t$ ordinary double points in total, corresponding to distinct roots $r_1,...,r_t$ of $Q(x)$, say. We conclude from the above that the normalisation of $C$ is given by a Weierstrass equation of the form 
\[y^2+\widetilde{Q}(x)y=\widetilde{P}(x)\] where $\widetilde{Q}(x)=Q(x)\prod_{i=1}^t(x-r_i)^{-1}$ and $\widetilde{P}(x)$ is a polynomial of degree at most $2(g-t)+2$. By Remark \ref{rem:size_of_2_tors_char_2} such a curve is ordinary if and only if $\widetilde{Q}(x)$ is separable, from which the result follows. 
 \end{proof}
 
  \section{Proofs of the main results}  \label{sec:main_proofs}
  
  Let $K$ be a finite extension of $\mathbb{Q}_2$,  $\mathcal{O}_K$ and $k$ its ring of integers and residue field respectively, and $v$ its normalised valuation.  

\subsection{Proofs of Theorem \ref{main:beyond_ordinary} and Proposition \ref{explicit_reduction_prop}}
  
\begin{lemma} \label{one_half_of_proof}
Suppose $|k|\geq g+1$ and that $C$ has reduction type $(\dagger)$. Then $C$ can be represented by a Weierstrass equation of the shape $(\star \star)$.
\end{lemma}

\begin{proof}
Let $\mathcal{C}/\mathcal{O}_K$ denote the stable model of $C$ and let $\iota$ denote the extension of the hyperelliptic involution to $\mathcal{C}$.  By  \cite[Appendice]{MR1106915}, the quotient $\mathcal{C}/\iota$ is a semistable model of $\mathbb{P}^1_K$ whose special fibre, by assumption,  consists of a single geometrically irreducible component  (if the geometric special fibre of $\mathcal{C}$ is a union of $2$ rational curves intersecting transversally then these are necessarily be swapped by $\iota$).  We conclude that the special fibre of $\mathcal{C}/\iota$ is isomorphic to $\mathbb{P}^1_k$. Thus $\mathcal{C}$ is a \textit{Weierstrass model} for $C$ in the sense of \cite[Definition 6]{MR1363944}, so can be represented by a Weierstrass equation \eqref{eq:weierstrass} over $\mathcal{O}_K$  (cf. \cite[Section 4.3]{MR1363944}).
Since $|k|\geq g+1$  we can moreover assume that the reduction $\overline{Q}(x)$ has degree $g+1$, and then that $Q(x)$ is monic. Applying Lemma \ref{sep_Q_char_2} to the reduced Weierstrass equation we see that  $\overline{Q}(x)$ is separable. 
  
  Completing the square in \eqref{eq:weierstrass} and scaling $y$ by 2 shows that $C$ can be represented by the  Weierstrass equation  $y^2 =Q(x)^2+4P(x)$. Write $Q(x)^2+4P(x)=cf(x)$ for $f(x)$  monic of degree $2g+2$, noting that   $c\equiv 1\bmod 4$ and that $f(x)$ is in $\cO_K[x]$. Write  $\gamma_1,\ldots, \gamma_{g+1}$ for the roots of $Q(x)$, all of which lie in in $K^{nr}$ since $\overline{Q}(x)$ is separable. Since $f(x) \equiv Q(x)^2 \bmod 4$, by Hensel's lemma for lifting coprime factorisations  (see \cite[ III.4.3 Theorem 1]{MR782296}) we can factor $f(x)$ over $K^{nr}$ as
$$
 f(x) = \prod_{i=1}^{g+1}f_i(x) 
$$
with each $f_i(x) \in \mathcal{O}_{K^{nr}}[x]$ monic quadratic satisfying $f_i(x)\equiv  (x-\gamma_i)^2 \bmod 4$. Since $f_i(x+\gamma_i)\equiv x^2 \bmod 4$ the discriminant of $f_i(x)$ is congruent to $0$ modulo $16$, so factoring $f_i(x)=(x-\alpha_i)(x-\beta_i)$ over $\bar{K}$ we find $v(\alpha_i-\beta_i)\geq v(4)$. Finally, since  $\overline{f_i}(x)=(x-\overline{\gamma_i})^2$ we have $\overline{\alpha_i}=\overline{\gamma_i}=\overline{\beta_i}$ for each $i$, hence separability of $\overline{Q}(x)$ gives $v(\alpha_i-\alpha_j)=v(\beta_i-\beta_j)=v(\alpha_i-\beta_j)=0$ for  $i\neq j$. Thus $y^2=cf(x)$ is a Weierstrass equation for $C$ of the desired form.
\end{proof}

To complete the proof of Theorem \ref{main:beyond_ordinary} it remains to study hyperelliptic curves $C$ given by a Weierstrass equation of the form $(\star \star)$. It will be convenient to introduce the following notation. 

\begin{notation}[cf. Notation \ref{Weierstrass_notat_intro}] \label{main_formulae_notat}
Suppose that $C$ is given by a Weierstrass equation of the form $(\star \star)$.  Define 
\[f_i(x)=(x-\alpha_i)(x-\beta_i),\quad \gamma_i=\frac{\alpha_i+\beta_i}{2},\quad \eta_i=\big(\frac{\alpha_i-\beta_i}{4}\big)^2.\]
As explained in Notation \ref{Weierstrass_notat_intro} we have $\eta_i,\gamma_i\in \mathcal{O}_{K^{nr}}$. Further, we have 
\begin{equation} \label{completed_square}
f(x)=\prod_{i=1}^{g+1}f_i(x)\quad \textup{ and }\quad f_i(x)=(x-\gamma_i)^2-4\eta_i.
\end{equation}
Next, set 
\[Q(x)=\prod_{i=1}^{g+1}(x-\gamma_i)\quad \textup{ and } P(x)=\frac{1}{4}(cf(x)-Q(x)^2).\]
By \eqref{completed_square} we have $f(x)\equiv Q(x)^2 \pmod 4$ so, since $c\equiv 1\pmod 4$, we have $P(x)\in \mathcal{O}_K[x]$. Finally,  write 
\[a=\frac{1}{4}(c-1)\in \mathcal{O}_K\quad \textup{ and }\quad r_i=\overline{\gamma_i}\in \bar{k}.\]
Note that  $\alpha_i-\gamma_i=\frac{\beta_i-\alpha_i}{2}\equiv 0\bmod 2$  so that $\gamma_i \equiv \alpha_i \bmod 2$. In particular, the $r_i$ are all distinct, and   $\overline{Q}(x)$ is separable. 
\end{notation}

\begin{lemma} \label{second_half_of_proof}
Suppose that $C$ is given by a Weierstrass equation of the form $(\star \star)$. Then $C$ has semistable reduction and, with $P(x)$ and $Q(x)$ as defined in Notation \ref{main_formulae_notat}, the stable model of $C$ is the $\mathcal{O}_K$-scheme defined by the Weierstrass equation $y^2+Q(x)y=P(x)$. Moreover, $C$ satisfies $(\dagger)$. 
\end{lemma}

\begin{proof}
Reversing the change of variables described in the proof of  Lemma \ref{one_half_of_proof} we see that $C/K$ is represented by the integral Weierstrass equation $y^2+Q(x)y=P(x)$. As explained in Notation \ref{main_formulae_notat} above,  $\overline{Q}(x)$ has degree $g+1$ and is separable. The result now follows from Lemma \ref{sep_Q_char_2}.
\end{proof}

Theorem \ref{main:beyond_ordinary} now follows from combining Lemma \ref{one_half_of_proof} and Lemma \ref{second_half_of_proof}.

\begin{proof}[Proof of Proposition \ref{explicit_reduction_prop}]
Let $C$ be given by a Weierstrass equation of the form $(\star \star)$. See Notation \ref{main_formulae_notat} for the quantities appearing below. As in Lemma \ref{second_half_of_proof}, the stable model $\mathcal{C}/\mathcal{O}_K$ of $C$  is given by the Weierstrass equation $y^2+Q(x)y=P(x)$. Denote by $\mathcal{C}'=\mathcal{C}\times_{\mathcal{O}_K}\mathcal{O}_{K^{nr}}$ the base-change of $\mathcal{C}$ to $\mathcal{O}_{K^{nr}}$. Then (as in the proof of Lemma  \ref{sep_Q_char_2}) $\mathcal{C}'$ is smooth away from the points  
$P_{i}=(r_i,\sqrt{P(r_i)})$   ($1\leq i\leq g+1$) on its the special fibre. 

For each $i$ we now compute the completed local ring $\widehat{\mathcal{O}}_{\mathcal{C}',P_{i}}$ of $\mathcal{C}'$ at $P_{i}$. Reordering the roots and translating $x$ we suppose $i=1$ and $\gamma_i=0$. Write $Q(x)=x\widetilde{Q}(x)$.  Further, write $f(x)=f_1(x)\widetilde{f}(x)$ where $\widetilde{f}(x)=\prod_{j=2}^{g+1}f_j(x)$. As in \eqref{completed_square} we have $f_1(x)=x^2-4\eta_1$, whilst $\widetilde{f}(x)\equiv \widetilde{Q}(x)^2 \pmod 4$. 

\begin{claim}
There is  $G(x)\in \mathcal{O}_{K^{nr}}[[x]]^\times$ with  
\[G(x)\equiv \widetilde{Q}(x)\hspace{-3pt}\pmod 2\quad \textup{ and } \quad G(x)^2=c\widetilde{f}(x).\]
\end{claim}

 \begin{proof}[Proof of claim]
Let $R=\mathcal{O}_{K^{nr}}[[x]]$,  write $c\widetilde{f}(x)=\widetilde{Q}(x)^2+4\widetilde{P}(x)$ for some $\widetilde{P}(x)\in \mathcal{O}_{K^{nr}}[x]$, and define $h(t)=t^2-t\widetilde{Q}-\widetilde{P}\in R[t]$. Since $R$ is Henselian and $\widetilde{Q}(0)$ is a unit in $\mathcal{O}_{K^{nr}}$, there is $G_0\in R$ with $h(G_0)=0$. Then $G=2G_0-\widetilde{Q}$ has the required properties.
\end{proof}
 
 With $G(x)$ as in the claim, set $u=G(x)^{-1}\big(y+x\cdot\frac{\widetilde{Q}(x)-G(x)}{2}\big)$. A straightforward computation gives
 \[y^2+Q(x)y-P(x)=G(x)^2(u^2+ux+\eta_1).\] 
 It follows that $\widehat{\mathcal{O}}_{\mathcal{C}',P_{r_i}}$ is isomorphic to the completed local ring of the scheme 
 \begin{equation*} \label{analytic_iso_scheme}
 u^2+ux+\eta_1=0 \subseteq \mathbb{A}^2_{\mathcal{O}_{K^{nr}}}
 \end{equation*}
 at the closed point  $x=0$, $u=\overline{\eta_1}$ on its special fibre.   If $v(\alpha_1-\beta_1)=v(4)$ then $\overline{\eta_1}\neq 0$ and this scheme is smooth. On the other hand, if $v(\alpha_1-\beta_1)>v(4)$ then $\overline{\eta_1}=0$ and, setting $v=u+x$, we see that the sought completed local ring  is  isomorphic to
 \[\widehat{\mathcal{O}_{K^{nr}}}[[u,v]]/\big(uv-\eta_1\big).\]
 We thus see that $P_{1}$ has thickness $v(\eta_1)=2(v(\alpha_1-\beta_1)-v(4))$ in $\mathcal{C}'$.
\end{proof}
%

\begin{remark} \label{explicit_reduction_formulae_rem}
Suppose that $C$  is given by a Weierstrass equation of the form $(\star \star)$. By Lemma \ref{second_half_of_proof} the special fibre of the stable model of $C$ is given by the Weierstrass equation $y^2+\overline{Q}(x)y=\overline{P}(x)$. With the quantities as in Notation \ref{main_formulae_notat} we have $\overline{Q}(x)=\prod_{i=1}^{g+1}(x-r_i)$. The polynomial $\overline{P}(x)$ is given explicitly by the formula 
\begin{equation} \label{eq:reduced_P}
\overline{P}(x)=\overline{a}\overline{Q}(x)^2+\sum_{i=1}^{g+1}\overline{\eta_i}\prod_{j\neq i}(x-r_j)^2.
\end{equation}
To see this note that from \eqref{completed_square} we have
\[cf(x)=(1+4a)\prod_{i=1}^{g+1}\big((x-\gamma_i)^2-4\eta_i \big).\]
Expanding out the righthand side, using the definition of $P(x)$, and reducing to the residue field recovers \eqref{eq:reduced_P}. 
\end{remark}

\subsection{Frobenius action on the special fibre}
It follows readily from the explicit equations for the special fibre of the stable model given in Remark \ref{explicit_reduction_formulae_rem}   that the Frobenius action takes the form asserted in Proposition \ref{prop:frob_action_intro}.

\begin{proof}[Proof of Proposition \ref{prop:frob_action_intro}]
As in the statement of the proposition we assume that $C$ is given by a Weierstrass equation of the shape $(\star \star)$. Let $\mathcal{C}/\mathcal{O}_K$ denote its stable model, as described explicitly in Remark \ref{explicit_reduction_formulae_rem}. 
 From the proof of Proposition \ref{explicit_reduction_prop} we see that the bijection  between ordinary double points on $\mathcal{C}_{\bar{k}}$ and pairs $\{\alpha_i,\beta_i\}$ of roots of $f(x)$ with $v(\alpha_i-\beta_i)>v(4)$  sends a pair $\{\alpha_i,\beta_i\}$ to the point $P_{i}=(r_i,\sqrt{P(r_i)})$ on the  special fibre $y^2+\overline{Q}(x)y=\overline{P}(x)$ of $\mathcal{C}$.  
The $\textup{Gal}(\bar{k}/k)$-equivariance of the map $\{\alpha_i,\beta_i\}  \mapsto P_i$ is clear, and we note also that $k(P_i)=k(r_i)$.  From the explicit normalisation map \eqref{eq:partial_normalisation} we see that $P_{i}$ is a split ordinary double point over $k(P_i)$  if and only if the quadratic equation 
\begin{eqnarray*}  
w^2+\overline{Q}'(r_i)w&=&\frac{\overline{P}(x)+\overline{P}(r_i)+\overline{Q}(x)\sqrt{\overline{P}(r_i)}}{(x-r_i)^2}\Bigg|_{x=r_i}\\
&\stackrel{\eqref{eq:reduced_P}}{=}& \overline{a}Q'(r_i)^2+\sum_{j\neq i}\overline{\eta_j}\prod_{s\neq i,j}(r_i-r_s)^2
\end{eqnarray*}
is soluble over $k(r_i)$. This latter condition is equivalent to \eqref{split_d_p_invariant_intro}, as can be seen by replacing $w$ by $Q'(r_i)w$ and considering the trace to $\mathbb{F}_2$ of the resulting equation.
\end{proof}

\subsection{Reduction map on $2$-torsion}

Now suppose that $C$ is given by a Weierstrass equation $y^2=cf(x)$ satisfying $(\star)$. As in Section \ref{2-tors_reducton_intro} we denote by $\mathcal{R}$ the set of roots of $f(x)$ in $\bar{K}$. For each even sized subset $S\subseteq \mathcal{R}$ we let $D_S$ denote the divisor defined in \eqref{eq:2-tors_divisor_intro}, whose class defines a $2$-torsion point on the Jacobian $J$ of $C$.
%

\begin{proof}[Proof of Proposition \ref{prop:reduction_map_2_tors}]
Let $P(x)$ and $Q(x)$ be as in Notation \ref{main_formulae_notat}. By Lemma \ref{second_half_of_proof} the $\mathcal{O}_K$-scheme $\mathcal{C}$ defined by the integral Weierstrass equation  $y^2+Q(x)y=P(x)$ realises the good reduction of $C$. In particular, by \cite[Theorem 9.5.1]{MR1045822}  the N\'{e}ron model of $J$ agrees with the identity component of the relative Picard functor $\textup{Pic}^0_{\mathcal{C}/\mathcal{O}_K}$. For $S\subseteq \mathcal{R}$ with $|S|$ even it follows that the reduction map $J(\bar{K})\rightarrow \overline{J}(\bar{k})$ sends (the class of) the divisor $D_S$ to the (class of the) divisor $ \sum_{r\in S}P_{\bar{r}}-\frac{|S|}{2}(\infty^+ +\infty^-)$ on $\mathcal{C}_{\bar{k}}$, where here $P_{\bar{r}}=(\bar{r},\sqrt{P(\bar{r})})$. The result now follows from Lemma \ref{lem:2-tors_char_2}.
 \end{proof}

\subsection*{Acknowledgments}

The first author is supported by a Royal Society Research Fellowship. The second author is supported by the Engineering and Physical Sciences Research Council (EPSRC) grant EP/V006541/1 `Selmer groups, Arithmetic Statistics and Parity Conjectures'.

\bibliographystyle{plain}

\bibliography{references}

\def\Dbar{\leavevmode\lower.6ex\hbox to 0pt{\hskip-.23ex \accent"16\hss}D}
  \def\cfac#1{\ifmmode\setbox7\hbox{$\accent"5E#1$}\else
  \setbox7\hbox{\accent"5E#1}\penalty 10000\relax\fi\raise 1\ht7
  \hbox{\lower1.15ex\hbox to 1\wd7{\hss\accent"13\hss}}\penalty 10000
  \hskip-1\wd7\penalty 10000\box7}
  \def\cftil#1{\ifmmode\setbox7\hbox{$\accent"5E#1$}\else
  \setbox7\hbox{\accent"5E#1}\penalty 10000\relax\fi\raise 1\ht7
  \hbox{\lower1.15ex\hbox to 1\wd7{\hss\accent"7E\hss}}\penalty 10000
  \hskip-1\wd7\penalty 10000\box7}
\begin{thebibliography}{10}

\bibitem{hyperuser}
A.~J. Best, L.~A. Betts, M.~Bisatt, R.~van Bommel, V.~Dokchitser, O.~Faraggi,
  S.~Kunzweiler, C.~Maistret, A.~Morgan, S.~Muselli, and S.~Nowell.
\newblock A user's guide to the local arithmetic of hyperelliptic curves.
\newblock {\em To appear in Bulletin of the London Mathematical Society}.

\bibitem{MR1045822}
Siegfried Bosch, Werner L{\"u}tkebohmert, and Michel Raynaud.
\newblock {\em N\'eron models}, volume~21 of {\em Ergebnisse der Mathematik und
  ihrer Grenzgebiete (3) [Results in Mathematics and Related Areas (3)]}.
\newblock Springer-Verlag, Berlin, 1990.

\bibitem{MR782296}
Nicolas Bourbaki.
\newblock {\em \'{E}l\'{e}ments de math\'{e}matique}.
\newblock Masson, Paris, 1985.
\newblock Alg\`ebre commutative. Chapitres 1 \`a 4. [Commutative algebra.
  Chapters 1--4], Reprint.

\bibitem{MR3576328}
Irene~I. Bouw and Stefan Wewers.
\newblock Computing {$L$}-functions and semistable reduction of superelliptic
  curves.
\newblock {\em Glasg. Math. J.}, 59(1):77--108, 2017.

\bibitem{MR3290950}
Wouter Castryck, Marco Streng, and Damiano Testa.
\newblock Curves in characteristic 2 with non-trivial 2-torsion.
\newblock {\em Adv. Math. Commun.}, 8(4):479--495, 2014.

\bibitem{DDMM18}
Tim Dokchitser, Vladimir Dokchitser, C\'{e}line Maistret, and Adam Morgan.
\newblock Arithmetic of hyperelliptic curves over local fields.
\newblock {\em To appear in Mathematische Annalen}.

\bibitem{DM2019}
Vladimir Dokchitser and Celine Maistret.
\newblock Parity conjecture for abelian surfaces.
\newblock {\em Preprint}, {arxiv: 1911.04626}, 2019.

\bibitem{MR2964027}
Igor~V. Dolgachev.
\newblock {\em Classical algebraic geometry: a modern view}.
\newblock Cambridge University Press, Cambridge, 2012.

\bibitem{MR3095219}
Arsen Elkin and Rachel Pries.
\newblock Ekedahl-{O}ort strata of hyperelliptic curves in characteristic 2.
\newblock {\em Algebra and Number Theory}, 7(3):507--532, 2013.

\bibitem{MR4201122}
Omri Faraggi and Sarah Nowell.
\newblock Models of hyperelliptic curves with tame potentially semistable
  reduction.
\newblock {\em Trans. London Math. Soc.}, 7(1):49--95, 2020.

\bibitem{MR1671741}
Ivan Kausz.
\newblock A discriminant and an upper bound for {$\omega^2$} for hyperelliptic
  arithmetic surfaces.
\newblock {\em Compositio Math.}, 115(1):37--69, 1999.

\bibitem{MR2272976}
Claus Lehr and Michel Matignon.
\newblock Wild monodromy and automorphisms of curves.
\newblock {\em Duke Math. J.}, 135(3):569--586, 2006.

\bibitem{MR1363944}
Qing Liu.
\newblock Mod\`eles entiers des courbes hyperelliptiques sur un corps de
  valuation discr\`ete.
\newblock {\em Trans. Amer. Math. Soc.}, 348(11):4577--4610, 1996.

\bibitem{MR1917232}
Qing Liu.
\newblock {\em Algebraic geometry and arithmetic curves}, volume~6 of {\em
  Oxford Graduate Texts in Mathematics}.
\newblock Oxford University Press, Oxford, 2002.
\newblock Translated from the French by Reinie Ern{\'e}, Oxford Science
  Publications.

\bibitem{MR1195511}
Paul Lockhart.
\newblock On the discriminant of a hyperelliptic curve.
\newblock {\em Trans. Amer. Math. Soc.}, 342(2):729--752, 1994.

\bibitem{MR1962052}
Michel Matignon.
\newblock Vers un algorithme pour la r\'{e}duction stable des rev\^{e}tements
  {$p$}-cycliques de la droite projective sur un corps {$p$}-adique.
\newblock {\em Math. Ann.}, 325(2):323--354, 2003.

\bibitem{MR0369362}
Yukihiko Namikawa and Kenji Ueno.
\newblock The complete classification of fibres in pencils of curves of genus
  two.
\newblock {\em Manuscripta Math.}, 9:143--186, 1973.

\bibitem{MR1740984}
Bjorn Poonen and Michael Stoll.
\newblock The {C}assels-{T}ate pairing on polarized abelian varieties.
\newblock {\em Ann. of Math. (2)}, 150(3):1109--1149, 1999.

\bibitem{MR1106915}
Michel Raynaud.
\newblock {$p$}-groupes et r\'{e}duction semi-stable des courbes.
\newblock In {\em The {G}rothendieck {F}estschrift, {V}ol. {III}}, volume~88 of
  {\em Progr. Math.}, pages 179--197. Birkh\"{a}user Boston, Boston, MA, 1990.

\bibitem{PS2019}
Padmavathi Srinivasan.
\newblock Conductors and minimal discriminants of hyperelliptic curves: A
  comparison in the tame case.
\newblock {\em Preprint}, arXiv:1910.08228, 2019.

\end{thebibliography}

\vspace{-0.6pt}
\end{document}